\documentclass[12pt,a4paper]{article}

\usepackage{amssymb,amsmath,amsfonts,amsthm,enumerate}
\usepackage{graphicx}

\usepackage[utf8]{inputenc}
\usepackage[english]{babel}

\numberwithin{equation}{section}
\newtheorem{theorem}{Theorem}[section]
\newtheorem{lemma}[theorem]{Lemma}

\theoremstyle{definition}
\newtheorem{assumption}[theorem]{Assumption}
\newtheorem{definition}[theorem]{Definition}

\title{Boundary value problem for a global in time \\ parabolic equation
\footnote{This work was supported by Russian Science Foundation (grant No.~19-11-00069).}}
\author{V.N.~Starovoitov\\
\small Lavrentyev Institute of Hydrodynamics, Novosibirsk, Russian Federation\\
\small E-mail:\; starovoitov@hydro.nsc.ru}
\date{}

\begin{document}

\maketitle

\abstract{
The aim of this paper is to draw attention to an interesting semilinear parabolic equation that arose when describing the
chaotic dynamics of a polymer molecule in a liquid. This equation is nonlocal in time and contains a term, called the
interaction potential, that depends on the time-integral of the solution over the entire interval of solving the problem.
In fact, one needs to know the ``future'' in order to determine the coefficient in this term, i.e., the causality principle is violated.
The existence of a weak solution of the initial boundary value problem is proven. The interaction potential
satisfies fairly general conditions and can have an arbitrary growth at infinity. The uniqueness of this solution
is established with restrictions on the length of the considered time interval.
}

\bigskip\noindent
\textbf{Key words:} nonlocal in time parabolic equation, initial boundary value problem, solvability, uniqueness

\bigskip\noindent
\textbf{2010 Mathematics Subject Classification:} 35K58

\section{Introduction}\label{Star-sec1}
Let $\varOmega$ be a bounded domain in $\mathbb{R}^n$, $n\ge 2$, with a Lipschitz boundary $\partial\varOmega$.
In the space-time cylinder $\varOmega_T=\varOmega\times (0,T)$, $T\in(0,\infty)$, we consider the following differential equation:
\begin{equation}\label{Star-1.1}
\partial_t u -\Delta u +\varphi\Big(\int_0^T u(\cdot,s)\,ds\Big)\, u =0,
\end{equation}
where $u=u(x,t)$ is an unknown scalar function, $x=(x_1,\ldots,x_n)$ the vector of the spatial variables
in $\mathbb{R}^n$, $t$ the time variable in the interval $[0,T]$, $\varphi$ a scalar
function that will be specified below.  We suppose that
the following boundary and initial conditions are satisfied:
\begin{equation}\label{Star-1.2}
u(x,t)=0\quad \text{for}\quad x\in\partial\varOmega,\quad t\in [0,T],
\end{equation}
\begin{equation}\label{Star-1.3}
u(x,0)=u_0(x)\quad \text{for}\quad x\in\varOmega,
\end{equation}
where the function $u_0:\varOmega\to \mathbb{R}$ is prescribed.

An interesting feature of this problem is that equation \eqref{Star-1.1} contains a non-local in time term
that depends on the integral over the whole interval $(0,T)$ on which the problem is being solved.
For this reason, equation \eqref{Star-1.1} is called global in time in the title of the paper.
There are a lot of works that study problems with memory for parabolic equations
which includes the integral of the solution from the initial to the current time and
it is not difficult to find appropriate works on this subject.
The problems with memory differ from ours. In fact, we need to know the ``future'' in order to determine
the coefficient in equation  \eqref{Star-1.1}. It should be noted that the problem cannot be reduced
to known ones by any transformations. There are papers that study problems, where the ``future'' stands in the data
(see, e.g., \cite{Star-P,Star-Sh,Star-BB}).
The paper \cite{Star-L} is devoted to the investigation of a system of equations
that contain an integral of the solution over the entire time interval, but this nonlocality is easily eliminated and the equation is
reduced to a parabolic equation with a prescribed combination of initial and final data as in \cite {Star-Sh}.

Our problem appeared when describing the chaotic dynamics of a single polymer molecule or, as it is also called, a polymer chain
in an aqueous solution (see \cite{Star-StSt}). The time $t$ in equation \eqref{Star-1.1} is in fact the arc length parameter along the chain.
The unknown function $u=u(x,t)$ is the density of probability that the $t$-th segment of the chain is at the point $x$.
Since each segment of the chain interacts with all other segments through the surrounding fluid,
the equation contains an interaction term which includes an integral of $u$ over the entire chain.
Equation \eqref{Star-1.1} is simpler than that obtained in \cite{Star-StSt}, however,
it looks similar and also contains the term with the integral of
the solution from $0$ to $T$.

In \cite{Star-St}, the weak solvability of the problem is proven for the case where $u$ is a positive bounded function
and $\varphi$ is the so called Flory --- Huggins potential. The positiveness is a natural requirement since
$u$ is the density of probability. The Flory --- Huggins potential is a convex increasing function that tends
to infinity as its argument approaches a certain positive value. Such an equation can appear in other problems
as well. It would be interesting to investigate it with another potential $\varphi$.
In this paper, we consider the potential $\varphi$ that is, in general, not convex and not
everywhere increasing. Besides that, we do not require that the solution is positive and bounded.

Generally speaking, equation \eqref{Star-1.1} has features unusual for parabolic equations.
First of all, the causality principle is violated. The state of the system depends not only on the past but also on
the future. Besides, from mathematical point of view, the solution of a nonlinear parabolic problem is commonly
being constructed locally in time and is extended afterwards. In our case this procedure is impossible.
Finally, as a rule, the local in time uniqueness of the solution implies the global one. We cannot prove the
uniqueness without restrictions on $T$. However, it is possible that a more skilled author will be able to do this.

In the next section, we define the notion of weak solution of problem \eqref{Star-1.1}--\eqref{Star-1.3}
and formulate Theorem~\ref{Star-t2.3}, the main result of the paper, that states the weak solvability of the problem.
The proof of this result (Section~\ref{Star-sec4}) is based on the Tikhonov theorem on the
existence of a fixed point of a map $\varPsi$. The construction of this map is divided into two standard problems
which are considered in Section~\ref{Star-sec3}. The mapping $\varPsi$ must be weakly continuous.
Roughly speaking, we have to show that not only a subsequence but the whole sequence converges weakly.

In Section~\ref{Star-sec5}, we present one of possible uniqueness results.
We managed to prove the uniqueness of the weak solution of problem \eqref{Star-1.1}--\eqref{Star-1.3}
only for sufficiently small $T$. Even in the case where $\varphi$ is the Flory --- Huggins potential,
a convex increasing function, and the solution of the problem is a non-negative bounded function,
we are forced to impose a restriction on $T$. Generally speaking, this fact can be explained by physical reasons.
Recall that the original problem describes the dynamics of a polymer chain and $T$ is its length.
If the chain is too long, it can form knots that significantly affect the chaotic motion of the chain.
Nevertheless, from a mathematical point of view, such a situation cannot be considered completely satisfactory.
In forthcoming studies, we will try to get rid of the restriction on $T$ in proving the uniqueness of the solution.
Notice that the Cauchy problem for the corresponding ordinary differential equation has a unique solution
for all values of $T$.
Since we intend to establish only a local in time uniqueness result, it makes no sense to prove it under the most
general conditions on the data of the problem. We suppose that the initial value of the solution is a bounded function.
This condition is natural for the problem of the polymer chain dynamics. Besides that, we impose an additional restriction
on the potential $\varphi$.

\section{Weak statement of the problem and main result}\label{Star-sec2}
At first, we formulate conditions on the potential $\varphi$ which will be fulfilled throughout the paper.
\begin{assumption}\label{Star-t2.1}
The potential $\varphi:\mathbb{R}\to [0,+\infty)$ is a continuous non-negative function
such that $\varphi(0)=0$ and $s\mapsto \varphi(s)s$ is a non-decreasing differentiable function whose
derivative is bounded on every compact subset of $\mathbb{R}$.
\end{assumption}
This assumption admits functions $\varphi$ that are not convex and not increasing as its argument tends
to $+\infty$. Besides that, we do not impose any restrictions on the growth of $\varphi$ at infinity.
Figure~\ref{Star-pic} shows examples of possible functions $\varphi$. The functions in the figure are even
but it is not necessary.
\begin{figure}[h]
\begin{center}
\includegraphics[width=0.7\textwidth]{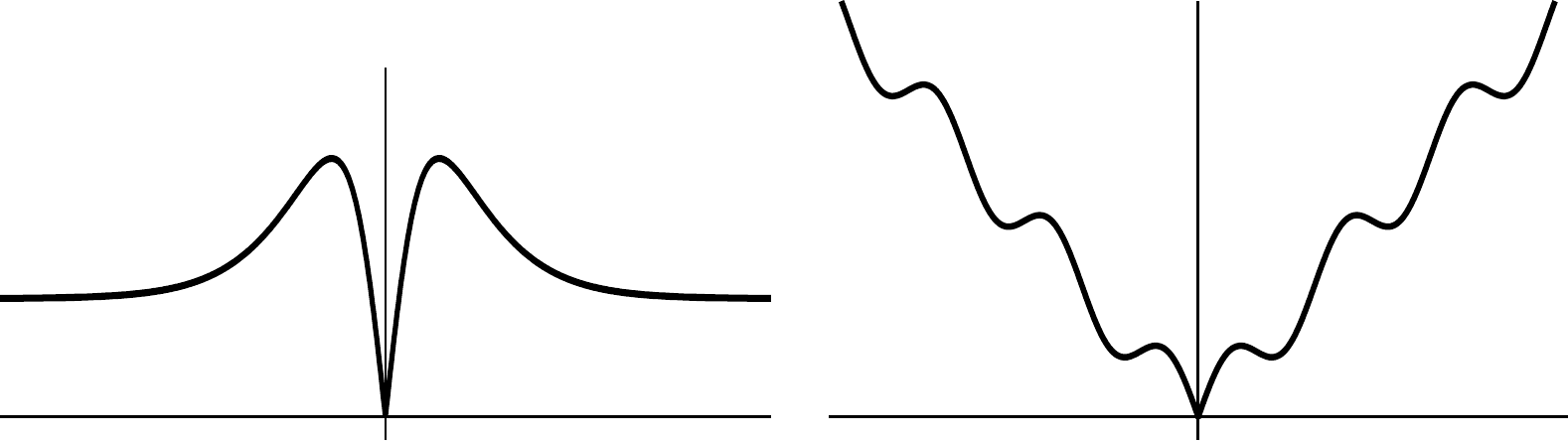}
\caption{\small Examples of the potential: $\displaystyle\varphi(s)=\frac{2|s|+\sinh|s|}{\cosh|s|}$ and
$\varphi(s)=2|s|+3\sin|s|$.}
\label{Star-pic}
\end{center}
\end{figure}

We will use the standard Lebesgue and Sobolev spaces $L^p(\varOmega)$, $H^1_0(\varOmega)$, $L^2(0,T;H^1_0(\varOmega))$
and $C(0,T; L^2(\varOmega))$ (see, e.g., \cite{Star-GGZ,Star-Br}).  As usual, $H^{-1}(\varOmega)$ is the dual space of $H^1_0(\varOmega)$
with respect to the pivot space $L^2(\varOmega)$. The norm and the inner product in $L^2(\varOmega)$ will be denoted by $\|\cdot\|$
and $(\cdot,\cdot)$, respectively.

\begin{definition}\label{Star-t2.2}
Let $\varphi$ satisfies Assumption~\ref{Star-t2.1} and $u_0\in L^2(\varOmega)$. A function
$u:\varOmega_T\to \mathbb{R}$ is said to be a weak solution of problem \eqref{Star-1.1}--\eqref{Star-1.3}, if
\begin{enumerate}
\item
$u\in L^2(0,T;H^1_0(\varOmega))$ and $\varphi(v)\,u\in L^1(\varOmega_T)$, where $v=\int_0^T u\,dt$;
\item
the following integral identity
$$
\int_0^T \int_\varOmega \big(u\,\partial_t h - \nabla u\cdot\nabla h - \varphi(v)\,u\, h\big)\,dxdt +
\int_\varOmega u_0 h_0\, dx=0
$$
holds for an arbitrary smooth in the closure of $\varOmega_T$ function $h$ such that $h(x,t)=0$
for $x\in\partial\varOmega$ and for $t=T$. Here, $h_0=h|_{t=0}$.
\end{enumerate}
\end{definition}

The main result of the paper is the following theorem.
\begin{theorem}\label{Star-t2.3}
If $u_0\in L^2(\varOmega)$, $T$ is an arbitrary positive number, and
$\varphi$ satisfies Assumption~\ref{Star-t2.1}, then problem \eqref{Star-1.1} -- \eqref{Star-1.3} has a weak solution $u\in
L^\infty(0,T;L^2(\varOmega))\cap L^2(0,T;H^1_0(\varOmega))$ such that
$$
\varphi(v)\in L^2(\varOmega),\quad \varphi(v)\, v\in L^2(\varOmega), \quad \varphi(v)\, u^2 \in L^1(\varOmega_T),
\quad u\in C(0,T;L^2(\varOmega)),
$$
where $v=\int_0^T u\,dt$.
\end{theorem}
The proof of this result will be given in Sections~\ref{Star-sec3} and \ref{Star-sec4}.

\section{Auxiliary results}\label{Star-sec3}
In this section, we consider two intermediate problems. The first problem
is elliptic and the second one is parabolic.

\subsection{Elliptic problem}\label{Star-sec3.1}
For every function $f:\varOmega\to \mathbb{R}$, define a function $v:\varOmega\to \mathbb{R}$ as a solution
of the following problem:
\begin{equation}\label{Star-3.1}
-\Delta v + \varphi(v)\,v +f=0\quad\text{in}\quad\varOmega,\quad v|_{\partial\varOmega}=0.
\end{equation}
This problem is a result of the integration of equation \eqref{Star-1.1} with respect to $t$ from $0$ to $T$.
The functions $v$ and $f$ correspond to $\int_0^T u(\cdot,t)\,dt$ and $u(\cdot,T)-u_0$, respectively.
Problem \eqref{Star-3.1} was already considered even in a more general situation with a nonlinear elliptic operator
instead of the Laplace operator (see \cite{Star-Gos,Star-Hess,Star-BrBr}). For every function $f\in H^{-1}\varOmega)$,
the existence of the solution $v\in H^1_0(\varOmega)$ such that $\varphi(v)v$ and $\varphi(v)v^2$
are in $L^1(\varOmega)$ was proven. The techniques used were different. In \cite{Star-Gos}, the Orlicz spaces were
employed under an additional assumption that the function $\varphi$ is even. In \cite{Star-Hess}, the problem was regularized
by a problem of a higher order. The order of the approximate problem depended on the dimension $n$ and
was such that its solution was a bounded function. So, there were no difficulties with the integrability of the term $\varphi(v)v$.
Notice that this is the main difficulty of the problem. In \cite{Star-BrBr}, the function $\varphi(v)v$ was
truncated by constants $\pm k$ and then $k$ tended to infinity.

In our case, the function $f$ is not only in $H^{-1}\varOmega)$ but in $L^2(\Omega)$, therefore,
we can expect more from the solution of the problem. We need in particular that $\varphi(v)\in L^2(\varOmega)$.
For brevity, we introduce the following notation: $\eta(s)=\varphi(s)s$. A function $v\in H^1_0(\varOmega)$ is said to
be a weak solution of problem \eqref{Star-3.1} if $\eta(v)\in L^1(\varOmega)$ and
\begin{equation}\label{Star-3.2}
(\nabla v,\nabla\psi)+(\eta(v),\psi)+(f,\psi)=0\quad\text{for all}\quad\psi\in H^1_0(\varOmega)\cap L^\infty(\varOmega).
\end{equation}

\begin{lemma}\label{Star-t3.1}
Let Assumption~\ref{Star-t2.1} be satisfied.
For every $f\in L^2(\varOmega)$, problem \eqref{Star-3.1} has a unique weak solution $v\in H^1_0(\varOmega)$ such that
\begin{enumerate}
\item
$\|\nabla v\|\le d(\varOmega)\,\|f\|$, where $d(\varOmega)$ is the diameter of the domain $\varOmega$;
\item
$\|\varphi(v)\,v\|\le \|f\|$;
\item
$\|\varphi(v)\|\le C$, where the constant $C$ depends on $\|f\|$ and $\varOmega$.
\end{enumerate}

\end{lemma}
\begin{proof}
As noted above, the weak solvability of problem \eqref{Star-3.1} is already known
(see, e.g., \cite{Star-Gos,Star-Hess,Star-BrBr}). We present here a simple proof of this fact that based
on the Galerkin method. Let $\{\psi_k\}$ be the orthonormal basis in $L^2(\varOmega)$ that consists of
the eigenfunctions of the Laplace operator $(-\Delta)$ with the homogeneous boundary condition.
If we define the inner product in $H^1_0(\varOmega)$ as $(w_1,w_2)_{H^1_0(\varOmega)}=(\nabla w_1,\nabla w_2)$ for
$w_1,w_2\in H^1_0(\varOmega)$, then the set $\{\psi_k\}$ is an orthogonal basis in this space.
It is well known that $\psi_k\in L^\infty(\varOmega)$ for every $k\in \mathbb{N}$ (see, e.g., \cite[Sec.~9.8]{Star-Br}).
Let $H_k$ be the subspace of $L^2(\varOmega)$ spanned by the basis functions $\{\psi_1,\ldots,\psi_k\}$.
For every $k\in \mathbb{N}$, denote by $v_k$ a function from $H_k$ such that
\begin{equation}\label{Star-3.3}
(\nabla v_k,\nabla\psi)+(\eta(v_k),\psi)+(f,\psi)=0\quad\text{for all}\quad\psi\in H_k.
\end{equation}
By employing the Brouwer fixed point theorem, it is not difficult to prove the existence of $v_k$.
Notice that the function $\eta(v_k)$ is bounded and
$(\eta(v_k),\psi)$ is well defined. If we take $\psi=v_k$, then we easily find that
\begin{equation}\label{Star-3.4}
\|\nabla v_k\|^2 +\int_\varOmega \varphi(v_k)v_k^2\,dx\le \|f\|\,\|v_k\|,\quad k\in \mathbb{N}.
\end{equation}
The positiveness of $\varphi$ and the Poincar\'{e} inequality imply that
\begin{equation}\label{Star-3.5}
\|\nabla v_k\|\le d(\varOmega)\,\|f\|,\quad k\in \mathbb{N},
\end{equation}
where $d(\varOmega)$ is the diameter of the domain $\varOmega$.

Therefore, the sequence  $\{v_k\}$ has a subsequence which converges weakly in $H^1_0(\varOmega)$
and almost everywhere in $\varOmega$ to a function $v$. We denote this subsequence again by $\{v_k\}$.
Due to the continuity of the function $\eta$, we obtain that
\begin{equation}\label{Star-3.6}
\eta(v_k)\to\eta(v)\quad\text{almost everywhere in $\varOmega$.}
\end{equation}
Let us prove that the functions $\eta(v_k)$ are uniformly integrable. Estimates \eqref{Star-3.4} and \eqref{Star-3.5}
imply that $\int_\varOmega \eta(v_k)\, v_k\, dx \le C_0$, where the constant $C_0$ is independent of $k$ and
depends only on $\|f\|$.
For every measurable set $A\subset\varOmega$ and every positive number $M$, we introduce the set
$A_M^k=\{x\in A\,|\, |v_k(x)|\ge M\}$. Then
$$
\int_{A_M^k}|\eta(v_k)|\,dx \le \frac{1}{M}\int_{\varOmega}\eta(v_k)\, v_k\, dx\le \frac{C_0}{M}.
$$
Since the function $\eta$ is continuous, there exists a constant $\gamma(M)$ such that $|\eta(s)|\le \gamma(M)$ for $s\in[-M,M]$.
In fact, as $\eta$ is non-decreasing, $\gamma(M)=\max \{-\eta(-M),\eta(M)\}$. Thus,
$$
\int_{A\setminus A_M^k} |\eta(v_k)|\,dx \le \gamma(M)\,\mu(A),
$$
where $\mu(A)$ is the Lebesgue measure of the set $A$. These inequalities imply that
$$
\int_A|\eta(v_k)|\,dx\le \frac{C_0}{M} + \gamma(M)\,\mu(A).
$$
For arbitrary $\varepsilon >0$, we take $M=2C_0/\varepsilon$ and $\delta=\varepsilon/(2\gamma(M))$.
Then we find that
$\int_A|\eta(v_k)|\,dx< \varepsilon$ for an arbitrary measurable set $A\subset\varOmega$ such that $\mu(A)<\delta$.
Thus, the uniform integrability of $\eta(v_k)$ is proven.

This fact together with \eqref{Star-3.6} and the Vitali convergence theorem (see, e.g., \cite[Sec.~4.8.7]{Star-MP})
enable us to conclude that $\eta(v)\in L^1(\varOmega)$ and
$\eta(v_k)\to \eta(v)$ in $L^1(\varOmega)$ as $k\to\infty$.
Now we are able to pass to the limit in \eqref{Star-3.3} as $k\to\infty$. As a result, we find that
$v$ satisfies \eqref{Star-3.2}, which means that $v$ is a weak solution of problem~\eqref{Star-3.1}.
This solution is unique since the function $\eta$ is non-decreasing.

Let us prove the estimates for $v$ stated in the lemma. The first estimate is a direct consequence of
\eqref{Star-3.5}. In order to prove the second one, we introduce the truncated function
for every $m\in \mathbb{N}$:
$$
v_m(x)=\begin{cases}
m, & v(x)\ge m,
\\
v(x), & -m< v(x)< m,
\\
-m, & v(x)\le -m.
\end{cases}
$$
Let us take $\psi=\eta(v_m)$ in \eqref{Star-3.2}. Since $\eta(v)\eta(v_m)\ge \eta^2(v_m)$, we easily find that
$$
\int_\varOmega \eta'(v_m)\,|\nabla v_m|^2\,dx +\|\eta(v_m)\|^2 \le \|f\| \|\eta(v_m)\|.
$$
The sequence $\{\eta^2(v_m)\}$ converges to $\eta^2(v)$ almost everywhere in $\varOmega$, therefore, the second
estimate of the lemma follows from the Fatou lemma.

Finally, let us prove the third estimate. If $A_1=\{x\in\varOmega\;|\; |v(x)|\ge 1\}$, then
$$
\int_{A_1} \varphi^2(v)\,dx \le \int_\varOmega \varphi^2(v)\,v^2\,dx\le \|f\|^2.
$$
Since the function $\varphi$ is continuous, there exists a constant $\gamma$ such that $\varphi^2(s)\le \gamma$ for $s\in[-1,1]$. Thus,
$$
\int_{\varOmega\setminus A_1} \varphi^2(v)\,dx \le \gamma\,\mu(\varOmega).
$$
These inequalities imply the required estimate with $C^2=\|f\|^2 + \gamma\,\mu(\varOmega)$.
\end{proof}

Denote by $V$ the mapping from $L^2(\varOmega)$ into $H^1_0(\varOmega)$ such that $v=V(f)$ is the unique
weak solution of problem \eqref{Star-3.1}.
\begin{lemma}\label{Star-t3.2}
Let $\{f_k\}$ be a sequence in $L^2(\varOmega)$ that converges to $f$ weakly in this space.
If $v_k=V(f_k)$ and $v=V(f)$, then
\begin{enumerate}
\item
$v_k\to v$ in $H^1_0(\varOmega)$ as $k\to\infty$;
\item
$\varphi(v_k)\to \varphi(v)$ weakly in $L^2(\varOmega)$ as $k\to\infty$.
\end{enumerate}
\end{lemma}
\begin{proof}
It is not difficult to see that $v_k-v$ is the weak solution of the following problem:
$$
-\Delta(v_k- v) + \varphi(v_k)\,v_k -\varphi(v)\,v +f_k-f=0\quad\text{in}\quad\varOmega,\quad (v_k-v)|_{\partial\varOmega}=0.
$$
Since the function $s\mapsto\eta(s)=\varphi(s)s$ is non-decreasing, we obtain that
$$
\int_\varOmega |\nabla (v_k-v)|^2\,dx \le \Big|\int_\varOmega (f_k-f)(v_k-v)\,dx\Big|
\le \|f_k-f\|_{H^{-1}(\varOmega)}\|v_k-v\|_{H^1_0(\varOmega)}.
$$
The first assertion of the lemma follows from the fact that
$\|f_k-f\|_{H^{-1}(\varOmega)}\to 0$ as $k\to\infty$.

Since $v_k\to v$ in $H^1_0(\varOmega)$ as $k\to\infty$, the sequence $\{v_k\}$ converges to $v$ in measure
and, as $\varphi$ is a continuous function, the sequence $\{\varphi(v_k)\}$ converges in measure to $\varphi(v)$.
Lemma~\ref{Star-t3.1} states that the sequence $\{\varphi(v_k)\}$ is bounded in $L^2(\varOmega)$, therefore,
due to the Vitali convergence theorem (see, e.g., \cite[Sec.~4.8.7]{Star-MP}), $\varphi(v_k)\to \varphi(v)$ in $L^p(\varOmega)$ as $k\to\infty$
for $p\in [1,2)$. Consequently,
$$
\int_\varOmega \big(\varphi(v_k)-\varphi(v)\big)\, h\, dx \to 0\quad\text{as}\quad k\to\infty
$$
for every $h\in L^\infty(\varOmega)$. The density of $L^\infty(\varOmega)$ in $L^2(\varOmega)$ implies the second
assertion of the lemma.
\end{proof}

The advantage of the lemma just proven is that we have established the convergence results not
for a subsequence but for the entire sequence $\{V(f_k)\}$. These results will be used for the proof
of the weak continuity of the mapping $\varPsi$ in the Tikhonov theorem.

\subsection{Parabolic problem}\label{Star-sec3.2}
We consider the following parabolic problem:
\begin{equation}\label{Star-3.7}
\partial_t u -\Delta u +\zeta u=0\quad\text{in}\quad\varOmega_T,\quad u|_{\partial\varOmega}=0,\quad u|_{t=0}=u_0,
\end{equation}
where $\zeta:\varOmega\to \mathbb{R}$ is an independent of $t$ non-negative function.
We suppose that $\zeta\in L^2(\varOmega)$ and $u_0\in L^2(\varOmega)$.
This problem is standard and we omit the proof of its unique weak solvability as well as various justifications
(see, e.g., \cite{Star-GGZ}). A little trouble is that the function $\zeta$ is in $L^2(\varOmega)$ only, which
can be easily overcome if we consider the solution in $L^2\big(0,T;H_0^1(\varOmega)\cap L^2(\varOmega,\zeta)\big)$,
where $L^2(\varOmega,\zeta)$ is the Hilbert space with the norm
$\|u\|_{\zeta}^2=\int_\varOmega \zeta u^2\,dx$.
The weak solution of problem \eqref{Star-3.3} satisfies the energy estimate:
\begin{equation}\label{Star-3.8}
\frac{1}{2}\, \|u(\cdot,s)\|^2 +\int_0^s \|\nabla u\|^2\,dxdt + \int_0^s \int_\varOmega \zeta\, u^2\,dxdt
\le\frac{1}{2}\, \|u_0\|^2
\end{equation}
for almost all $s\in[0,T]$. Besides that, $\partial_t u$ belongs to the space $L^2\big(0,T;(H_0^1(\varOmega)\cap L^2(\varOmega,\zeta))^*\big)$,
where $(H_0^1(\varOmega)\cap L^2(\varOmega,\zeta))^*$ is the conjugate space to $H_0^1(\varOmega)\cap L^2(\varOmega,\zeta)$.
As a consequence of this fact, we find that $u\in C(0,T;L^2(\varOmega))$.
Thus, the function $u_T=u|_{t=T}$ is well defined as an element of $L^2(\varOmega)$ and \eqref{Star-3.8} holds for all $s\in [0,T]$.

For every non-negative function $\zeta\in L^2(\varOmega)$, we denote by $U(\zeta)$ the unique weak solution
of problem \eqref{Star-3.7} and by $U_T(\zeta)$ the function $U(\zeta)|_{t=T}$.
Our goal is to investigate the dependence of $U$ and $U_T$ on $\zeta$.

\begin{lemma}\label{Star-t3.3}
Let $u_0\in L^2(\varOmega)$ and a sequence of non-negative functions $\{\zeta_k\}$ converges weakly in $L^2(\varOmega)$
to a function $\zeta$. Then $U_T(\zeta_k)\to U_T(\zeta)$ weakly in $L^2(\varOmega)$ as $k\to\infty$.
\end{lemma}
\begin{proof}
Notice that the lemma asserts the convergence of the entire sequence $\{U_T(\zeta_k)\}$ .
For brevity, we denote by $u_k$ and $u$ the functions $U(\zeta_k)$ and $U(\zeta)$, respectively.
Let $h:\varOmega_T\to \mathbb{R}$ be an arbitrary smooth function such that $h|_{\partial\varOmega}=0$.
As it follows from \eqref{Star-3.8},
$$
\int_0^T\int_\varOmega |\nabla u_k|^2\,dxdt\le \frac{1}{2}\, \|u_0\|^2,\quad k\in \mathbb{N}.
$$
This estimate implies that
$$
\big\|\nabla\int_0^T u_k h\,dt\big\|\le C,\quad k\in \mathbb{N},
$$
where the constant $C$ depends, of course, on $h$. Therefore, the sequence $\{u_k\}$ has a subsequence
$\{u_{k'}\}$ such that
\begin{align*}
& u_{k'}\to w\quad \text{weakly in}\quad L^2\big(0,T;H_0^1(\varOmega)),
\\
& \int_0^T u_{k'} h\,dt\to \int_0^T w h\,dt\quad\text{in}\quad L^2(\varOmega)
\end{align*}
as $k'\to\infty$, where $w$ is some function. As a consequence of the second relation, we have that
$$
\int_0^T\int_\varOmega \zeta_{k'} u_{k'} h\,dxdt \to \int_0^T\int_\varOmega \zeta w h\,dxdt
\quad\text{as}\quad k'\to\infty.
$$
Here, we have used the fact that the functions $\zeta_k$ do not depend on $t$.
The passage to the limit as $k'\to\infty$ in the weak formulation of \eqref{Star-3.3} and the uniqueness of the
solution of this problem imply that $w=U(\zeta)$.
Besides that, equation \eqref{Star-3.7} implies that
\begin{multline*}
\int_\varOmega \big(U_T(\zeta_{k'})-U_T(\zeta)\big)\, h(\cdot,T)\, dx
\\
=-\int_0^T\int_\varOmega \big(\nabla (u_{k'}-u)\cdot\nabla h + (\zeta_{k'}u_{k'}-\zeta u)\,h\big)\,dxdt
\to 0\quad\text{as}\quad k'\to\infty.
\end{multline*}
Since the set of smooth functions is dense in $L^2(\varOmega)$, we
conclude that $U_T(\zeta_{k'})\to U_T(\zeta)$ weakly in $L^2(\varOmega)$ as $k'\to\infty$.
Thus, we have proven that every subsequence of the sequence $\{U_T(\zeta_k)\}$ has a subsequence that
converges weakly to $U_T(\zeta)$ in $L^2(\varOmega)$. The uniqueness of the limit yields
the assertion of the lemma.
\end{proof}

\section{Weak solvability of the problem}\label{Star-sec4}
In order to prove the weak solvability of problem \eqref{Star-1.1}--\eqref{Star-1.3}
stated in Theorem~\ref{Star-t2.3}, we employ the Tikhonov fixed-point theorem which states that, for a reflexive separable
Banach space $X$ and a closed convex bounded set $E\subset X$, if a mapping $\varPsi:E\to E$ is
weakly sequentially continuous, then $\varPsi$ has at least one fixed point in $E$.

We take $X=L^2(\varOmega)$, $E=\{w\in L^2(\varOmega)\,|\; \|w\|\le \|u_0\|\}$ and define the mapping $\varPsi$ as follows:
for every $w\in E$, $\varPsi (w)= U_T(\varphi(v))$, where $v=V(w-u_0)$. The operators $V$ and $U_T$ were introduced
in the previous section. If $u$ is the weak solution of the original problem, then $\int_0^T u\, dt= V(u_T-u_0)$ and
$$
u_T=U_T\Big(\varphi\big(\int_0^T u\, dt\big)\Big)=\varPsi(u_T).
$$
Thus, the fixed point of the mapping $\varPsi$ is a function that corresponds to $u_T$. If we know $u_T$, we define
the function $v=V(u_T-u_0)$ and finally find the weak solution of the original problem
$u=U(\varphi(v))$.

As it follows from the properties of the operators $V$ and $U$, the mapping $\varPsi$ is well defined on $L^2(\varOmega)$.
Besides that, for every non-negative function $\zeta\in L^2(\varOmega)$, the function $U(\zeta)$ satisfies \eqref{Star-3.8} which implies
that $\|U_T(\zeta)\|\le \|u_0\|$. Thus, $\varPsi(E)\subset E$. It remains to prove the weak sequential continuity
of $\varPsi$.

Let $\{w_k\}$ be an arbitrary sequence in $E$ that converges to $w\in E$ weakly in $L^2(\varOmega)$.
We need to prove that $\varPsi(w_k)\to \varPsi(w)$ weakly in $L^2(\varOmega)$ as $k\to\infty$.
It is a simple consequence of the results obtained in Section~\ref{Star-sec3}.
Due to Lemma~\ref{Star-t3.2}, $\varphi(v_k)\to\varphi(v)$ weakly in $L^2(\varOmega)$ as
$k\to\infty$, where $v_k=V(w_k-u_0)$ and $v=V(w-u_0)$. In turn, Lemma~\ref{Star-t3.3} implies
that $\varPsi (w_k)= U_T(\varphi(v_k))\to U_T(\varphi(v))=\varPsi (w)$ weakly in $L^2(\varOmega)$
as $k\to\infty$.
Thus, the weak solvability of problem \eqref{Star-1.1}--\eqref{Star-1.3} is proven.

\section{Uniqueness of the solution}\label{Star-sec5}
In this section, we prove the following uniqueness theorem.
\begin{theorem}\label{Star-t5.1}
Let Assumption~\ref{Star-t2.1} be satisfied.
Assume that there exist constants $K$ and $M$ such that $|u_0|\le K$ almost everywhere in $\varOmega$
and $|\varphi'(s)|\le M$ for $s\in [-KT,KT]$.

If $MKT^2<2$, then the weak solution of problem \eqref{Star-1.1}--\eqref{Star-1.3}
is unique.
\end{theorem}
\begin{proof}
Suppose that this problem has two weak solutions $u_1$ and $u_2$.
Denote by $u$ its difference $u_1-u_2$. Then
$$
\partial_t u -\Delta u + \varphi (v_1)\, u_1 - \varphi(v_2)\,u_2=0,
\quad u|_{\partial\varOmega}=u|_{t=0}=0,
$$
where $v_i(x)=\int_0^T u_i(x,t)\,dt$, $i=1,2$.
Since equation~\eqref{Star-1.1} admits the maximum principle, $|u_1|\le K$ and $|u_2|\le K$ almost everywhere in $\varOmega_T$
and $\big|\int_0^T u_i\,dt\big|\le KT$, $i=1,2$, almost everywhere in $\varOmega$.
Therefore, the equation for $u$ holds in $L^2(0,T; H^{-1}(\varOmega))$.
The multiplication of this equation by $u$ and the integration over
$\varOmega$ lead to the following equality:
$$
\frac{1}{2}\,\frac{d}{dt}\|u\|^2+\|\nabla u\|^2+\int_\varOmega \varphi(v_1)\, u^2\,dx
+\int_\varOmega \big(\varphi(v_1)-\varphi(v_2)\big)\,u_2\, u\,dx =0
$$
which implies that
$$
\frac{1}{2}\,\|u(\cdot,t)\|^2+\int_0^t\|\nabla u(\cdot,s)\|^2ds\le
MK\int_0^t\int_\varOmega |v(x)|\,|u(x,s)|\,dxds
$$
for all $t\in[0,T]$, where $v=v_1-v_2$. Since
\begin{multline*}
\int_0^t\int_\varOmega |v(x)|\,|u(x,s)|\,dxds=\int_\varOmega |v(x)|\,\int_0^t|u(x,s)|\,ds dx
\\
\le \|v\|\,\int_0^t\|u(\cdot,s)\|\,ds\le \int_0^T\|u(\cdot,s)\|\,ds\,\int_0^t\|u(\cdot,s)\|\,ds,
\end{multline*}
we obtain the following inequality:
\begin{equation}\label{Star-5.1}
\|u(\cdot,t)\|^2\le 2MK\int_0^T\|u(\cdot,s)\|\,ds\,\int_0^t\|u(\cdot,s)\|\,ds
\end{equation}
which holds for all $t\in[0,T]$. If $MKT^2<2$, this inequality implies that $u=0$.
Really, let us denote: $\xi(t)=\int_0^t\|u(\cdot,s)\|\,ds$.
Then \eqref{Star-5.1} can be rewritten as follows:
$$
(\xi'(t))^2\le 2MK\xi(T)\, \xi(t),\quad t\in (0,T].
$$
This inequality yields that $\xi(t)\le t^2MK\xi(T)/2$ for all $t\in [0,T]$ and, in particular,
$\xi(T)\le T^2MK\xi(T)/2$. If $T^2MK<2$, then the only solution of this inequality is $\xi(T)=0$
which implies that $u=0$.
\end{proof}


\begin{thebibliography}{99}

\bibitem{Star-P}
Pao,~C.~V. (1995).  Reaction diffusion equations with nonlocal boundary and nonlocal initial conditions.
\emph{Journal of Mathematical Analysis and Applications}. 195(3): 702-718.

\bibitem{Star-Sh}
Shelukhin,~V.~V. (1991). A problem with time-averaged data for nonlinear parabolic equations. \emph{Siberian Mathematical Journal}.
32(2): 309–320.

\bibitem{Star-BB}
Buhrii,~O., Buhrii,~N. (2019). Nonlocal in time problem for anisotropic parabolic equations with
variable exponents of nonlinearities. \emph{J. Math.Anal.Appl.} 473(2): 695–711.

\bibitem{Star-L}
Lyubanova,~A.~Sh. (2015). On nonlocal problems for systems of parabolic equations.
\emph{J. Math. Anal. Appl.} 421(2): 1767–1778.

\bibitem{Star-StSt}
Starovoitov,~V.~N., Starovoitova,~B.~N. (2017). Modeling the dynamics of polymer chains in water solution. Application to sensor design.
\textit{Journal of Physics: Conference series}. 894: P.n.~012088.

\bibitem{Star-St}
Starovoitov,~V.~N. (2018). Initial boundary value problem for a nonlocal in time parabolic equation.
\emph{Siberian Electronic Mathematical Reports}. 15: 1311–1319.

\bibitem{Star-GGZ}
Gajewski,~H., Gr\"oger,~K., Zacharias,~K. (1974).
\emph{Nichtlineare Operatorgleichungen und Operatordifferentialgleichungen.}
Mathematische Lehrb\"ucher und Monographien, II. Abteilung, Mathematische Monographien, B.~38.

\bibitem{Star-Br}
Br{\'e}zis,~H. (2010). \emph{Functional analysis, Sobolev spaces and partial differential equations.}
Springer Science \& Business Media.

\bibitem{Star-Gos}
Gossez,~J.~P. (1971). Op{\'e}rateurs monotones non lin{\'e}aires dans les espaces de Banach non r{\'e}flexifs.
\emph{Journal of Mathematical Analysis and Applications}. 34(2): 371-395.

\bibitem{Star-Hess}
Hess,~P. (1973). A strongly nonlinear elliptic boundary value problem.
\emph{Journal of Mathematical Analysis and Applications}. 43(1): 241-249.

\bibitem{Star-BrBr}
Br{\'e}zis,~H., Browder,~F.~E. (1978). Strongly nonlinear elliptic boundary value problems.
\emph{Annali della Scuola Normale Superiore di Pisa, Classe di Scienze}. 5(3): 587-603.

\bibitem{Star-MP}
Makarov,~B., Podkorytov,~A. (2013). \emph{Real analysis: measures, integrals and applications.}
Springer Science \& Business Media.

\end{thebibliography}
\end{document}